\documentclass[a4paper,11pt]{article}

\usepackage{adjustbox}
\usepackage{multirow}

\usepackage[english]{babel}
\usepackage{cite}
\usepackage{amsfonts}
\usepackage{latexsym}
\usepackage{amsthm}
\usepackage{amsopn}
\usepackage{amsmath}
\usepackage[all]{xy}
\usepackage{verbatim}
\usepackage{amssymb}
\usepackage{mathrsfs}
\usepackage{float}
\usepackage[active]{srcltx}
\usepackage{fullpage}
\usepackage{hyperref}
\usepackage{manyfoot}
\usepackage{amscd}
\usepackage{amscd}
\usepackage{color}
\usepackage{cite}
\usepackage{tikz}
\usepackage{capt-of}
\usetikzlibrary{matrix}
\usepackage[hmargin=3cm, vmargin=3cm]{geometry}
\usepackage{enumerate}

\usepackage{color, xcolor}

\makeatletter
\newtheorem*{rep@theorem}{\rep@title}
\newcommand{\newreptheorem}[2]{%
\newenvironment{rep#1}[1]{%
 \def\rep@title{#2 \ref{##1}}%
 \begin{rep@theorem}}%
 {\end{rep@theorem}}}
\makeatother

\newtheorem*{theorem*}{Theorem}
\newtheorem{theorem}{Theorem}[section]
\newtheorem{proposition}[theorem]{Proposition}
\newtheorem{lemma}[theorem]{Lemma}

\newreptheorem{theorem}{Theorem}

\theoremstyle{definition}
\newtheorem{remark}[theorem]{Remark}

\newcommand{\NN}{\mathbb{N}}
\newcommand{\ZZ}{\mathbb{Z}}

\newcommand{\PP}{\mathbb{P}}
\newcommand{\oo}{\mathcal{O}}

\renewcommand{\to}{\longrightarrow}

\newcommand{\commentpigna}[1]{}

\DeclareMathOperator{\Sing}{Sing}

\title{New surfaces with canonical map of high degree}
\author{Christian Gleissner, Roberto Pignatelli, Carlos Rito}
\date{}
\begin{document}
\maketitle
%


\begin{abstract}
We give an algorithm that, for a given value of the geometric genus $p_g,$
computes all regular product-quotient surfaces with abelian group that have at most canonical singularities
and have canonical system with at most isolated base points. 
We use it to show that there are exactly two families of such surfaces with
canonical map of degree $32$.  We also construct a surface with $q=1$ and canonical map of degree $24$. 
These are regular surfaces with $p_g=3$ and base point free canonical system.
We discuss the case of regular surfaces with $p_g=4$ and base point free canonical system.
\end{abstract}


\Footnotetext{{}}{2010 \textit{Mathematics Subject Classification}:
Primary: 14J29 Secondary: 14J10}

\Footnotetext{{}} {\textit{Keywords}: Surface of general type; Product-quotient surface; Canonical map}




\section{Introduction} \label{sec:intro}

Let $S$ be a smooth surface of general type with irregularity $q$ and geometric genus $p_g\geq 3$.
Denote by $\phi$ the canonical map of $S$ and let $d:=\deg(\phi).$
It is known since Beauville \cite{Be} that if the canonical image $\phi(S)$ is a surface, then
$$d\leq 36-9q \ \ {\rm if} \ \ q\leq 3, \ \ \ d\leq 8 \ \ {\rm if} \ \ q\geq 4.$$
Beauville has also constructed families of examples with $\chi(\mathcal O_S)$ arbitrarily
large for $d=2, 4, 6, 8.$ Despite being a classical problem, for $d>8$ the
number of known examples is scarce. Tan's example \cite[\S 5]{Ta} with
$d=9, q=0$ and Persson's example \cite{Pe} with $d=16,$ $q=0$ are well known.
Du and Gao \cite{DG} show that if the canonical map is an abelian cover
of $\mathbb P^2,$ then these are the only possibilities for $d>8.$
More recently the third author has given examples with $d=16,$ $q=2$ \cite{Ri1} and $d=24,$ $q=0$ \cite{Ri2}.
There is a paper \cite{Ye} claiming the existence of the case $d=36,$ but, to our knowledge, the proof is not correct.

In this paper we consider the problem of finding  product-quotient surfaces $(A\times B)/G$ with at most canonical singularities
having canonical map of maximum degree. For these surfaces $K^2\leq 8\chi$ (see \cite{4names}),
equality holding if and only if the {\it quotient model} $(A\times B)/G$ is smooth, ${\it i.e.}$ the action of $G$ is free.
Here Beauville's argument gives 
 $$d\leq 32-8q \ \ {\rm if} \ \ q\leq 3,$$
equality holding if and only if $G$ acts freely, $p_g=3$ and the canonical system is base point free.
In order to be able to understand this system, we restrict our study to abelian groups $G.$
Such surfaces are then abelian coverings of the product $(A/G)\times(B/G),$
and we can use Pardini's \cite{Pa} formulas to understand their canonical curves.

We give an algorithm that, for a given value of the geometric genus $p_g$ and some $n\in\mathbb N,$
computes all regular product-quotient surfaces with abelian group
$G$ that have at most canonical singularities and have canonical system with at most $n$ base points.
Applying it to the case $K^2=32,$ we get exactly two families of surfaces with $p_g=3,$ $q=0$ and canonical map of degree $K^2=32$ onto $\mathbb P^2.$
We describe these surfaces as $(\mathbb Z/2)^4$-coverings of $\mathbb P^1\times\mathbb P^1$ in Section \ref{deg32}.

We have also found a family of product-quotient surfaces with $p_g=3,$ $q=1$ and canonical map of degree $K^2=24$ onto $\mathbb P^2.$
We give the construction as a $(\mathbb Z/2)^3$-covering of $E\times\mathbb P^1$ in Section \ref{deg24}, where $E$ is an elliptic curve.
One can show that this is the unique such family with group $G=(\mathbb Z/2)^3,$ we give the idea for the proof of this fact in Remark \ref{rem}.

For product-quotient surfaces with $p_g\geq 4$ and $q\leq 3,$
Beauville's proof gives the inequality $$d\leq 8 \left( 1 + \frac{3-q}{p_g-2}\right)\leq 20.$$
Strangely enough the output of our algorithm for $p_g=4$ does not contain any quotient
$(A \times B)/G$ with $G$ acting freely,
and therefore there exists no product-quotient surface $(A \times B)/G$ with $G$ abelian and canonical map of degree  $20$.
We show that the maximum degree for regular such surfaces is $12.$
The value $p_g=4$ is a surprising gap. Indeed Catanese constructed in \cite{Ca2} regular product-quotient surfaces with $p_g=5$ and $6$ of the form $(A \times B)/G$ with $G$ abelian acting freely and canonical system without base points. Catanese's examples, having canonical map of degree $1$ and $\phi(S)$ of very high degree,  have been an important source of inspiration for this paper.

To keep the paper as simple as possible, for the convenience of the readers, we describe our examples directly as abelian covers of $(A/G) \times (B/G)$ instead of as quotients $(A \times B)/G$. We refer the interested reader to \cite{BP12,BP16} and the references therein for the theory of product-quotient surfaces in the general case of arbitrary singularities.

An implementation of our algorithm as MAGMA script may be downloaded at\newline
\url{http://www.science.unitn.it/~pignatel/papers/CanonicalMapProg.magma}

\subsubsection*{Acknowledgments}
C. Gleissner wants to thank his mentor I. Bauer for the motivation to write an algorithm to classify 
regular product-quotient surfaces with abelian group and analyse their canonical system. 
He was partially supported by FIRB 2012 "Moduli spaces and Applications". 

R. Pignatelli is grateful to F. Catanese for inviting him to Bayreuth with the ERC-2013-Advanced Grant-340258-TADMICAMT; part of this research took place during his visit.
He  is  a  member  of  GNSAGA-INdAM and  was partially  supported  by  the  project  PRIN 2015 Geometria delle 
variet\`a algebriche.

C. Rito was supported by FCT (Portugal) under the project PTDC/MAT-GEO/2823/ 2014,
the fellowship SFRH/BPD/111131/2015 and by CMUP (UID/MAT/00144/2013),
which is funded by FCT with national (MEC) and European structural funds through the programs FEDER,
under the partnership agreement PT2020.


\section{$(\mathbb Z/2)^r$-coverings and canonical systems}

The following result is taken from \cite[Proposition 6.6]{Ca} (see also \cite{Pa}). 
\begin{proposition}\label{CatProp}
A normal finite $G\cong(\mathbb Z/2)^r$-covering $Y\rightarrow X$ of a smooth variety $X$ is completely
determined by the datum of
\begin{enumerate}
\item reduced effective divisors $D_{\sigma},$ $\forall\sigma\in G,$ with no common components;
\item divisor linear equivalence classes $L_1,\ldots,L_r,$ for $\chi_1,\ldots,\chi_r$ a basis of the dual group of characters
$G^{\vee},$ such that $$2L_i\equiv\sum_{\chi_i(\sigma)=1}D_{\sigma}.$$ 
\end{enumerate}
Conversely, given (1) and (2), one obtains a normal scheme $Y$ with a finite
$G\cong(\mathbb Z/2)^r$-covering $Y\rightarrow X,$ with branch curves the divisors $D_{\sigma}.$
\end{proposition}
The covering $\psi \colon Y\rightarrow X$ is embedded in the total space of the direct sum of the line bundles whose
sheaves of sections are the $\mathcal O_X(L_i),$ and is there defined by equations
$$u_{\chi_i}u_{\chi_j}=u_{\chi_i+\chi_j}\prod_{\chi_i(\sigma)=\chi_j(\sigma)=1}x_{\sigma},$$
where $x_{\sigma}$ is a section such that ${\rm div}(x_{\sigma})=D_{\sigma}.$

The scheme $Y$ is irreducible if $\{\sigma|D_{\sigma}>0\}$ generates $G.$

If the branch locus of $\psi$ is simple normal crossing, then the surface $Y$ is smooth and its invariants are
$$\chi(\mathcal O_Y)=2^r\chi(\mathcal O_X)+\frac{1}{2}\sum_{\chi\in G^{{\vee}*}} \left(L_{\chi}^2+K_XL_{\chi}\right),$$
$$p_g(Y)=p_g(X)+\sum_{\chi\in G^{{\vee}*}} h^0(X,\mathcal O_X(K_X+L_{\chi})).$$

For each $\sigma\in G,$ denote by $R_{\sigma}\subset Y$ the reduced divisor supported on $\psi^*(D_{\sigma}).$
We get from \cite[Proposition 4.1, c)]{Pa} and \cite[Proposition 2.1)]{infrig} that, if $X$ is Gorenstein, for any $\chi,$ $$(\psi_*\omega_Y)^{(\chi)} \cong \omega_X( L_{\chi}).$$ 
Combining with the Hurwitz formula, $$K_Y=\psi^*(K_X)+\sum_{\sigma\in G^*}R_{\sigma},$$
we obtain that the canonical linear system of $Y$ is generated by
\begin{equation}\label{cangens}
\psi^*|K_X+L_i| \sum_{\chi_i(\sigma)=0}R_{\sigma},\ \ i\in J,
\end{equation}
where $J:=\{j:|K_X+L_j|\ne\emptyset\}.$

\section{The families with $\deg(\phi)=32$}\label{deg32}

Let $f,g$ be the rational fibrations of $X:=\mathbb P^1\times\mathbb P^1,$
and let $F_1,\ldots,F_6$ be distinct fibres of $f$ and $E_1,\ldots,E_6$ be distinct fibres of $g.$
Denote by $e_1,\ldots,e_4$ the generators of $(\mathbb Z/2)^4.$ We set $e_{i_1\cdots i_r} := e_{i_1}+\cdots +e_{i_r} $.
\subsection{Building data $2\times(1,1,1,1,1,1)$}

Consider the $(\mathbb Z/2)^4$-covering $$\psi:Y\rightarrow X$$ given by
$$D_{e_1}:=F_1,\ D_{e_2}:=F_2,\ D_{e_3}:=F_3,\ D_{e_4}:=F_4, \ D_{e_{13}}:=F_5,\ D_{e_{24}}:=F_6,$$
$$D_{e_{234}}:=E_1, \ D_{e_{134}}:=E_2,\ D_{e_{124}}:=E_3,\ D_{e_{123}}:=E_4,\ D_{e_{14}}:=E_5,\ D_{e_{23}}:=E_6.$$

For $i,j,k,l\in {\mathbb Z/2},$ let $\chi_{ijkl}$ denote the character which takes the value $i,j,k,l$ on $e_1,e_2,e_3,e_4,$ respectively.
There exist divisors $L_{ijkl}$ such that
\begin{equation*}\label{eq1}
2L_{ijkl}\equiv\sum_{\chi_{ijkl}(\sigma)=1}D_{\sigma},
\end{equation*}
thus the covering $\psi$ is well defined.
Since there is no 2-torsion in the Picard group of $X,$ then $\psi$ is uniquely determined.
The surface $Y$ is smooth because the curves $D_{e_1},\ldots,D_{e_{234}}$ are smooth with pairwise transverse intersections only.

We have $$L_{1100}\equiv L_{0011}\equiv L_{1111}\equiv 2F+2E,$$
where $F$ is a fibre of $f$ and $E$ is a fibre of $g.$
For the remaining cases we have $$L_{ijkl}\equiv F+2E\ \ {\rm or}\ \ 2F+E.$$

This implies that 
$$\chi(\mathcal O_Y)=16+\frac{1}{2}\sum \left(L_{ijkl}^2+K_XL_{ijkl}\right)=4$$ and
$$p_g(Y)=0+\sum h^0(X,\mathcal O_X(K_X+L_{ijkl}))=3.$$

We get from (\ref{cangens}) that $K_Y$ is generated by  the following divisors,
respectively associated to the characters $\chi_{0011}, \chi_{1100}$ and $\chi_{1111}$:
$$ \widehat F_1+\widehat F_2+\widehat E_1+\widehat E_2,\ 
\widehat F_3+\widehat F_4+\widehat E_3+\widehat E_4,\ 
\widehat F_5+\widehat F_6+\widehat E_5+\widehat E_6, $$
where $\widehat F_i:=\frac{1}{2}\psi^*(F_i)$ and $\widehat E_i:=\frac{1}{2}\psi^*(E_i).$

The fact $\widehat F_i\widehat E_j = 4$ implies $K_Y^2=32.$

By looking to their images on $X,$ one verifies that the above three divisors have no common intersection.
Thus $|K_Y|$ is base-point free and then $K_Y^2>0$ implies that the canonical map of $Y$ is not composed with a pencil.
Hence its image is $\mathbb P^2,$ a surface of degree 1, therefore the degree formula implies that the canonical map of $Y$
is of degree $K_Y^2=32.$

\subsection{Building data $2\times(2,1,1,1,1)$}

Here we only give the building data of the covering, the verifications are analogous to the ones in the previous section. 

\[
D_{e_{1}}:=F_1,\ D_{e_{134}}:=F_2,\  D_{e_{123}}:=F_3+F_4,\ D_{e_{13}}:=F_5,\ D_{e_{14}}:=F_6,
\]
\[
D_{e_{2}}:=E_1,\ D_{e_{234}}:=E_2,\  D_{e_{124}}:=E_3+E_4,\ D_{e_{23}}:=E_5,\ D_{e_{24}}:=E_6.
\]

As in the previous case, setting  $\widehat F_i:=\frac{1}{2}\psi^*(F_i)$ and $\widehat E_i:=\frac{1}{2}\psi^*(E_i),$ $K_Y$
is generated by the divisors
$\widehat F_1+\widehat F_2+\widehat E_1+\widehat E_2,\ 
\widehat F_3+\widehat F_4+\widehat E_3+\widehat E_4,\ 
\widehat F_5+\widehat F_6+\widehat E_5+\widehat E_6, $
respectively associated to the characters $\chi_{0011}, \chi_{1100}$ and $\chi_{1111}$.

\section{A family with $\deg(\phi)=24,$ $q=1$}\label{deg24}

Let $$X:=E\times F,$$ with $F\cong\mathbb P^1$ and $E$ a smooth elliptic curve.
Let $E_1,\ldots,E_6\subset X$ be distinct elliptic fibres and $F_1,F_2,F_3\subset X$ be distinct rational fibres.
Since the sum of two points in an elliptic curve is divisible by $2$ in the Picard Group, there are fibres $F_{ij}$ such that
$2F_{ij}\equiv F_i+F_j,$ $i,j\in\{1,2,3\}.$

Let $e_1,e_2,e_3$ be the generators of $(\mathbb Z/2)^3$, set $e_{i_1\cdots i_r} := e_{i_1}+\cdots +e_{i_r} $ and consider the divisors
$$D_{e_1}:=E_1+E_2,\ D_{e_2}:=E_3+E_4,\ D_{e_3}:=E_5+E_6,$$
$$D_{e_{23}}:=F_1,\ D_{e_{13}}:=F_2,\ D_{e_{12}}:=F_3,$$
$$L_{100}:=E+F_{23},\ L_{010}:=E+F_{13}, \ L_{001}:=E+F_{12}.$$

For $i,j,k\in \mathbb Z/2,$ let $\chi_{ijk}$ denote the character which takes the value $i,j,k$ on $e_1,e_2,e_3,$ respectively.
The above data satisfies
\begin{equation*}\label{eq1}
2L_{ijk}\equiv\sum_{\chi_{ijk}(\sigma)=1}D_{\sigma},
\end{equation*}
thus from Proposition \ref{CatProp} it defines a $(\mathbb Z/2)^3$-covering $$\psi:Y\to X.$$
Note that there are four different possible choices for each $F_{ij}$: a different choice produces a different $Y$.
The surface $Y$ is smooth because the curves $D_{e_1},\ldots,D_{e_{23}}$ are smooth with pairwise transverse intersections only.

The fact $$L_{\chi+\eta}\equiv L_\chi+L_\eta -\sum_{\chi(\sigma)=\eta(\sigma)=1}D_\sigma$$
implies that
$$L_{111}\equiv 3E+T,$$
$$L_{110}\equiv 2E+F'_{12},\ L_{101}\equiv 2E+F'_{13}, \ L_{011}\equiv 2E+F'_{23},$$
where 
$$T:=F_{12}+F_{13}+F_{23}-F_1-F_2-F_3$$
and $F_{ij}'$ is a fibre linearly equivalent to $F_{ij}+T$. 

We notice that the divisor class $2T$ is trivial. We choose the $F_{ij}$ so that the divisor class $T$ is not trivial, so $T$ is a 2-torsion element of the Picard group.

Since $K_X\equiv -2E,$ we have that 
$$\chi(\mathcal O_Y)=0+\frac{1}{2}\sum \left(L_{ijk}^2+K_XL_{ijk}\right)=3,$$
$$p_g(Y)=0+\sum h^0(X,\mathcal O_X(K_X+L_{ijk}))=3,$$
and then $q(Y)=1.$

We get from (\ref{cangens}) that $K_Y$ is generated by the following divisors:
$$
\widehat E_5+\widehat E_6+\widehat F_3+\widetilde F'_{12},\
\widehat E_3+\widehat E_4+\widehat F_2+\widetilde F'_{13},\ 
\widehat E_1+\widehat E_2+\widehat F_1+\widetilde F'_{23},
$$
corresponding respectively to the characters
$\chi_{110}$,
$\chi_{101}$ and 
$\chi_{011}$, 
where $\widehat E_i:=\frac{1}{2}\psi^*(E_i),$ $\widehat F_i:=\frac{1}{2}\psi^*(F_i)$, $\widetilde F_{ij}:=\psi^*(F_{ij})$ and $\widetilde F'_{ij}:=\psi^*(F'_{ij}).$

The facts $\widehat E_i\widehat F_j=2$ and $\widehat E_i\widetilde F_{ij}=4$ imply $K_Y^2=24.$ 

The fibres $E_i$, $F_j$, $F'_{kl}$ are distinct with the only possible exceptions $F'_{ij}=2F_k,$ $\{i,j,k\}=\{1,2,3\}$.
Then the above three divisors have no common intersection since their images on $X$ have no common intersection. Thus $|K_Y|$ is base-point free and then, arguing as in Section \ref{deg32},
the canonical map of $Y$ is of degree $K_Y^2=24.$

\begin{remark}\label{rem}
We have a proof that these are the only irregular product-quotient surfaces of the form $(A \times B)/(\mathbb Z/2)^3$ with canonical map of degree $24$. We quickly sketch here the main point of the proof.

Such surfaces $S$ are $(\mathbb Z/2)^3-$covers of $E \times F$ ($E$ elliptic, $F$ rational) branched on an union of elliptic fibres $E_i$ and rational fibres $F_j$. Since  the action of $(\mathbb Z/2)^3$ on $A \times B$ is free, each $D_\sigma$ is either of the form $\sum E_i$ or of the form $\sum F_j$.

By (\ref{cangens}) the canonical system is generated by three divisors corresponding to three characters. If these characters are linearly independent we can assume w.l.o.g. that they are $\chi_{100}$, $\chi_{010}$ and $\chi_{001}$. Then $h^0(E\times F, K_{E\times F}+L_{100})=1$. It is easy to prove that the class of $K_{E\times F}+L_{100}$ can't be trivial, so it is the class of a rational fibre $F_1$, and analogous statement holds for $K_{E\times F}+L_{010}$ and $K_{E\times F}+L_{001}$. Then all the three divisors contain the pull-back of a rational fibre $F_i$. Then $\forall i \in \{1,2,3\}$, $D_{e_i}$ cannot contain  any elliptic fibre $E_j$ or there would be a base point of $K_S$ on $\widehat{E}_j$.
By Hurwitz formula one deduces $D_{e_{ij}} \equiv E$, $D_{e_{123}} \equiv 2E$. Since there is at least a rational fibre $F_0$ in the branch locus, w.l.o.g. $F_0 \leq D_{e_1}$ and one finds a base point of $K_S$ on $F_0 \cap D_{e_{23}}$, a contradiction.

So, the three characters are linearly dependent. The rest of the proof uses similar arguments.
\end{remark}

\section{The algorithm}
In this section we describe our algorithm, producing all regular product-quotient surfaces whose quotient model $Y:=(A \times B)/G$ has $G$ abelian,
at most  rational double points as singularities, and canonical system with at most isolated base points. 

By \cite[Remark 2.5]{4names} every singular point $y \in Y$ is then of type $A_{n_y}$, $n_y \in \NN$.

\begin{lemma}\label{newbound}
Let $Y:=(A \times B)/G$ be the quotient model of a product-quotient surface with only canonical singularities such that $G$ is abelian.
Set $g(A), g(B)$ for the genus of the curve $A, B,$ respectively, and assume w.l.o.g. $g(A) \geq g(B)$. Set also $\chi:=\chi(\oo(Y))$. Then
\[
g(B) \leq 1+2\chi+2\sqrt{\chi^2+2\chi}\ ,\ \ \ \ \ \ \ \ \ \ \ \ 
g(A) \leq 4\chi \frac{g(B)+1}{g(B)-1}+1\ .
\]
\end{lemma}
\begin{proof}
According to \cite[Proposition 3.10]{Pol09},
\[
\chi=\frac{(g(A)-1)(g(B)-1)}{|G|} +\frac1{12} \sum_{y \in \Sing Y} \frac{n_y^2-1}{n_y} \geq \frac{(g(A)-1)(g(B)-1)}{|G|} \ .
\]
Since $G$ is abelian, we have $|G| \leq 4g(B)+4$ by \cite[Corollary 9.6]{Br00}, which implies
$$
\chi (4g(B)+4) \geq \chi|G|\geq (g(A)-1)(g(B)-1) \geq (g(B)-1)^2 \ .
$$ In particular
\[
g(B)^2-(4\chi+2)g(B)+1-4\chi \leq 0.
\]
\end{proof}

We assume $Y$ regular, then $E:=A/G\cong F:=B/G \cong \PP^1$. Since $G$ is abelian, then the finite map $
\psi \colon Y \rightarrow E \times F \cong \PP^1 \times \PP^1$
 is a Galois cover with Galois group $G$. The branching locus of $\psi$ is the union of the lines $E_i:=E \times q_i$, $F_j:=p_j \times F$, where $p_j$ are the branching points of $A \rightarrow E$ and $q_i$ are the branching points of $B \rightarrow F$. The cover $A \rightarrow E$ associates naturally to each point $p_j$ its {\it local monodromy}, an element $g_j$ of $G$, that is also the local monodromy of $F_j$ for $\psi$. The element $g_j$ is the image of a small loop around $p_j$ for the map in \cite[page 1002]{4names}, not depending on the choice of the loop since $G$ is abelian. In the notation of \cite{Pa},  $F_j$ is a component of the divisor $D_{H,\eta}$ with $H=\langle g_j \rangle$ and $\eta \in H^*$ defined by $\eta(g_j)=e^\frac{2\pi i}{m_j}$ where $m_j$ is the order of $g_j$ in $G$.
 
By the Riemann Existence Theorem, the local monodromies give a bijection among the Galois covers of $\PP^1$ and the maps $\{p_j\} \rightarrow G$ such that $p_j$ is a finite subset of $\PP^1$ and the image is a set of generators of $G$ that is {\it spherical}, {\it i.e.} such that the sum of the images of the $p_j$ is zero.
So we produce regular product-quotient surfaces by producing two sets of spherical generators of $G$ and then choosing freely the points $p_j, q_i$.

The {\it type} of the set of generators is the {\it multiset} (a set whose elements are allowed to have a multiplicity in $\NN$) of the orders $m_j$ of the local monodromies of the $p_j$.
See \cite{4names} for details. 

Fix now $p_g(Y) \in \NN$. The algorithm is the following:
\begin{itemize}
\item[] $\mathbf 1^{st}$ {\bf Step}: Since $\chi=p_g(Y)+1$, Lemma \ref{newbound} determines finitely many possible pairs of genera $(g(A),g(B))$ and so finitely many possible orders of the group $|G|\leq 4g(B)+4$. The inequalities in \cite[Theorem 9.1, Corollary 9.6]{Br00} and \cite[Proposition 3.6]{Pol09} limit the types $T_1$ and $T_2$ of the coverings $A \rightarrow E$ and $B \rightarrow F$ to finitely many possibilities. Our program lists first all possible $5-$tuples $[g(A),g(B),|G|,T_1,T_2]$.
\item[] $\mathbf 2^{nd}$ {\bf Step}: For each resulting $5-$tuple $[g(A),g(B),|G|,T_1,T_2],$ the program searches among all groups of order $|G|$ for pairs of systems of spherical generators of respective types $T_1$ and $T_2$. For each such pairs it computes the singularities of the resulting surface $(A \times B)/G$ using \cite[Proposition 1.17]{BP12} (or the equivalent \cite[Proposition 3.3]{Pa}) and discards all pairs giving singularities not canonical.
\item[] $\mathbf 3^{rd}$ {\bf Step}: Finally the program discards, among the surfaces produced by Step 2, those whose canonical system has a 1-dimensional base component, as follows. Since the singularities of $Y$ are Gorenstein, we can use Pardini's formula  (\cite[Proposition 4.1, c)]{Pa} and \cite[Proposition 2.1)]{infrig} ) for splitting its canonical system  as in (\ref{cangens}).
More precisely we obtain subsystems of the form $\psi^*|M_\chi| +\Phi_\chi,$ $\chi \in G^*$, generating the canonical system, where $|M_\chi|$ is a (possibly empty) complete linear system on $\PP^1 \times \PP^1$ and $\Phi_\chi$ is an effective divisor supported on the union of the $E_i$ and the $F_j$. Since every complete linear system on $\PP^1 \times \PP^1$ is base point free, then the canonical system of the product-quotient surface has at most isolated base points if and only if the  divisors $\Phi_\chi$ such that $|M_\chi| \neq \emptyset$ meet only at a finite number of points.
\end{itemize}

The program returns: the group $G$, the types $T_i$, a pair of generating vectors, the systems $M_\chi$ that are not empty, the singularities of $Y$,
and the number of base points of the canonical system.

\begin{remark}
By Beauville's argument,
\[
\deg \phi \leq \frac{8(p_g+1)- b- \frac23 \sum_{y \in \Sing Y} \frac{n_y^2-1}{n_y}}{p_g-2},
\]
where $b$ is the number of base points. The equality holds if and only if $\phi(S) \subset \PP^{p_g-1}$ is a surface of minimal degree $p_g-2$.
\end{remark}

Running the program for $p_g=3,$ we obtain the following result.

\begin{proposition}
There are exactly $2$ families of regular product-quotient surfaces\\ $(A \times B)/G$ with $G$ abelian acting freely, $p_g=3$ and canonical system base point free,  the families described in Section \ref{deg32}.

There are further $17$ families of regular product-quotient surfaces $(A \times B)/G$ with $G$ abelian, $p_g=3$, canonical system base point free whose quotient model has only canonical singularities.
\end{proposition}

We notice that the two families are distinct even as families in the Gieseker moduli space of surfaces of general type. Indeed by \cite[Theorem 1.3]{FanoConference} a  surface in one family is not even deformation equivalent to any surface in the other family.

The degrees of the canonical maps of the surfaces in the $17$ further families form the set $\{2,4,6,8,16\}$.

\section{The case $p_g=4$}


Running the program for $p_g=4$ we get:

\begin{proposition}\label{proppg=4}
There are no regular product-quotient surfaces $(A \times B)/G$ with $G$ abelian acting freely, $p_g=4$ and canonical system base point free.

There are $60$ families of regular product-quotient surfaces $(A \times B)/G$ with $G$ abelian, $p_g=4$, canonical system base point free whose quotient model has only canonical singularities. 
\end{proposition}

The highest degree realized by the 60 families in Proposition \ref{proppg=4} is $12$, realized by a family with group $G=(\ZZ/3)^2$. The branching divisor is the union of $8$ lines, $4$ for each ruling: $E_1+E_2+E_3+E_4+F_1+F_2+F_3+F_4$, with local monodromies
\begin{align*}
E_1 \mapsto& (1,1,2)&
E_2 \mapsto& (2,2,0)&
E_3 \mapsto& (1,2,1)&
E_4 \mapsto& (2,1,0)\\
F_1 \mapsto& (2,1,1)&
F_2 \mapsto& (0,1,2)&
F_3 \mapsto& (1,2,1)&
F_4 \mapsto& (0,2,2)\\
\end{align*}
The surface has $9$ singular points of type $A_2$ and $K^2=24$. There are $4$ characters $\chi$ with $|M_\chi| \neq \emptyset,$ here are the respective $\Phi_\chi$:
\begin{align*}
\Phi_{(0,1,0)} =& \widehat{E}_1+\widehat{E}_4+\widehat{F}_1+\widehat{F}_2&
\Phi_{(1,0,1)} =& 2\widehat{E}_1+2\widehat{F}_1\\
\Phi_{(0,2,0)} =& \widehat{E}_2+\widehat{E}_3+\widehat{F}_3+\widehat{F}_4&
\Phi_{(2,2,2)} =& 2\widehat{E}_4+2\widehat{F}_2\\
\end{align*}
with $\widehat{E}_i=\frac13 \psi^*(E_i)$,  $\widehat{F}_i=\frac13 \psi^*(F_i)$.

Recall that since the canonical system is base point free and has positive self-intersection, the canonical map is not composed with a pencil.
Since $2\Phi_{(0,1,0)}=\Phi_{(1,0,1)} +\Phi_{(0,2,0)},$ the image of the canonical map is contained in a quadric cone and therefore is the quadric cone. More precisely, one can choose sections $x_0,x_1,x_2,x_3$ of $H^0(S,K_S)$ with respective divisors $ \Phi_{(0,1,0)}$, $\Phi_{(1,0,1)} $, $\Phi_{(0,2,0)}$, $\Phi_{(2,2,2)} $ such that the canonical image is the quadric $x_0^2=x_1x_2$.

There are three more families of surfaces in our list of  $60$ with $K^2_S\geq 24$: one with $K^2 =36$ and two with $K^2=32.$
We can show that their canonical image is not contained in a quadric, and therefore the maximal canonical degree we can reach for $p_g=4$ is $12$.

\bibliography{References}

\bigskip

\bigskip

\noindent Christian Gleissner \vspace{2mm}\\
University of Bayreuth, Lehrstuhl Mathematik VIII;\\
Universit\"atsstra\ss e 30, D-95447 Bayreuth, Germany\\
\verb|Christian.Gleissner@uni-bayreuth.de|\\

\noindent Roberto Pignatelli \vspace{2mm}\\ Dipartimento di Matematica, Universit\`{a} di
Trento \\ Via Sommarive, 14 I-38123 Trento (TN), Italy.  \\
\verb|Roberto.Pignatelli@unitn.it|\\

\noindent Carlos Rito \vspace{2mm}
\\{\it Permanent address:}
\\ Universidade de Tr\'as-os-Montes e Alto Douro, UTAD
\\ Quinta de Prados
\\ 5000-801 Vila Real, Portugal
\\ www.utad.pt, \verb|crito@utad.pt|
\\{\it Temporary address:}
\\ Departamento de Matem\' atica
\\ Faculdade de Ci\^encias da Universidade do Porto
\\ Rua do Campo Alegre 687
\\ 4169-007 Porto, Portugal
\\ www.fc.up.pt, \verb|crito@fc.up.pt|

\end{document}